\documentclass[onefignum,onetabnum]{siamart220329}

\usepackage{amsfonts}
\usepackage{graphicx}

\def\pra{Phys. Rev. A}
\def\prb{Phys. Rev. B}

\def\prl{Phys. Rev. Lett.}

\newcommand{\eps}{\varepsilon}

\headers{Robustness of Propagating  Bound States in the Continuum}{Lijun Yuan and Ya  Yan Lu} 

\title{On the Robustness of Propagating Bound States in the Continuum 
\thanks{Submitted to the editors DATE.
\funding{This work was supported by the National Natural Science Foundation of China, under project 12571384, and the Research Grants Council of Hong Kong Special Administrative Region, China, under project CityU 11317622. }}
} 

\author{Lijun Yuan\thanks{College of Mathematics and Statistics, Chongqing Technology and Business University, Chongqing, China   (\email{ljyuan@ctbu.edu.cn}).}
 \and Ya Yan Lu\thanks{Department of Mathematics, City University of  Hong Kong, Hong Kong   (\email{mayylu@cityu.edu.hk}).}}

\begin{document}

\maketitle

\begin{abstract}
  Bound states in the continuum (BICs) are localized eigenmodes with their
  frequencies in the radiation continuum of scattering states.  The
  existence of a BIC implies the loss of uniqueness for scattering
  problems with given incident waves. Perturbed wave systems close to
  the ideal ones with a BIC exhibit strong resonance effects that are
  essential to numerous practical applications. A question of
  fundamental importance is whether a BIC is robust, i.e., whether it
  can continue its existence
  when the structure is slightly perturbed. In an
  earlier work [Yuan and Lu, Optics Letters, Vol.~42, pp.~4490-4493,
  2017], for a class
  of BICs governed by the two-dimensional (2D) Helmholtz equation, which
  are not trivially protected by symmetry, we uncovered the conditions
  that ensure robustness and formally constructed the BIC 
  in perturbed systems using a perturbation method. In this paper, we
  present a rigorous theory on the robustness of BICs in 2D dielectric
  structures with a single periodic direction. Specifically, we 
  analyze the solvability and provide estimates for each order in the
  perturbation series, and prove the convergence of the series. 
\end{abstract}



\section{Introduction}

In classical and quantum linear wave systems involving at least one unbounded
spatial variable, a bound state in the continuum (BIC) is a nonzero
solution that decays to zero at 
infinity, but for the same frequency and wavevector (if appropriate),
the governing equations support waves that propagate to or from
infinity. The concept of BIC was originally introduced by von Neumann
and Wigner in 1929~\cite{neumann29}, for a 1D Schr\"odinger equation with an
oscillatory potential that decays to zero at infinity. Since then,
numerous BICs have been found in various wave
systems~\cite{Review16,Sadr21,kosh23}, and the
governing equations can be infinite linear systems, ordinary differential
equations~\cite{fonda63,still75,fried85,gomis17,xiuchen26}, Helmholtz
equations~\cite{bonnet94,evans94,shipman03,porter05,shipman07,mari08,bulg08}, Maxwell's
equations~\cite{hsu13_2,jin19}, and others~\cite{mciver96,linton97}. 
Mathematically, a BIC corresponds to a discrete eigenvalue
in the continuous spectrum. At the frequency and wavevector (if
appropriate) of the BIC, the scattering problems with given incoming
waves from infinity lose uniqueness~\cite{bonnet94,mciver96,linton97}. 

In recent years, BICs have found many applications in photonics~\cite{Review16,Sadr21,kosh23,azzam21,kang23}. Most
of these applications are related to resonances near BICs when the
structure or the wavevector is perturbed. If the structure is
translationally invariant or periodic in $x$ and $y$, where
$\{x,y,z\}$ is a Cartesian coordinate system,  a BIC is only
bounded in $z$ and has a specific
frequency $\omega_*$ and a specific (Bloch) wavevector $(\alpha_*, \beta_*)$. For
a real wavevector $(\alpha,\beta)$ near $(\alpha_*, \beta_*)$, the governing
equation usually has only a resonant mode which satisfies an outgoing
radiation condition as $z \to \infty$ and has a complex frequency
$\tilde{\omega}$. Scattering problems for incoming waves with the same
wavevector $(\alpha, \beta)$ and a real frequency $\omega$ near
$\mbox{Re}(\tilde{\omega})$ exhibit resonant wave phenomena, such as
local field enhancement and anomalous transmission and reflection, and
they are closely related to the quality-factor ($Q$ factor) $Q = -0.5
\mbox{Re}(\tilde{\omega})/\mbox{Im}(\tilde{\omega})$. Typically, the
$Q$ factor is proportional to $1/s^2$, where $s$ measures the
wavevector difference, but near some special BICs (the so-called super
BICs), the $Q$ factor may be proportional to $1/s^4$, $1/s^6$, etc~\cite{yuan20,nan25prl}. The super-BICs induce high-$Q$ resonances for a wider
range of wavevectors, and are particularly useful for practical
applications~\cite{kodi17,hwang21}. 

When the structure is perturbed, a BIC may or may not be
destroyed. For applications in photonics, BICs are usually studied in
lossless non-magnetic structures described by a real 
relative permittivity (also called dielectric function) $\epsilon$
that depends on the spatial variables. If the original structure with dielectric function
$\epsilon_*$ supports a BIC with frequency $\omega_*$ and wavevector
$(\alpha_*, \beta_*)$, we consider perturbed structures with the
dielectric function $\epsilon = \epsilon_* + \delta F$, where $F$ (the
perturbation profile) is a real function that preserves the
translational invariance or  periodicity in $x$ and $y$, and
$\delta$ is the amplitude of the perturbation. A natural question is
whether the perturbed structure supports a BIC (near the original one in
the unperturbed structure) with the same wavevector
$(\alpha_*, \beta_*)$. The answer depends on the nature of the BIC and
the symmetry of the perturbation. The so-called symmetry-protected
BICs have a symmetry mismatch with waves that propagate to/from
infinity, and will continue their existence in the perturbed structure
if the perturbation is sufficiently small and perserves the relevant
symmetry. 
The wavevector of the 
symmetry-protected BICs is fixed at $(\alpha_*, \beta_*) = {\bf
  0}$. If the perturbation breaks the relevant symmetry, the BIC is
usually destroyed and it becomes a resonant mode with a finite $Q$
factor. The $Q$ factor is usually proportional to $1/\delta^2$, but
special symmetry-breaking perturbations can give a $Q$ factor
proportional to $1/\delta^4$, $1/\delta^6$, etc~\cite{yuan20}.
If the symmetry-breaking perturbation contains parameters, the BIC may be
preserved by properly tuning the parameters.  For generic
perturbations, it is possible to identify the minimum number of
tunable parameters needed to preserve the BIC~\cite{lijun20pd1,amgad23pra}. 

If the original BIC has a nonzero wavevector $(\alpha_*, \beta_*)$,
there is usually no BIC (near the original one) in the perturbed
structure with the same
wavevector. Like the symmetry-protected BICs, we can
study the $Q$ factor of resonant modes in the perturbed structure, 
and tune parameters in the perturbation to preserve the BIC with the
same wavevector $(\alpha_*, \beta_*)$. Moreover, it is important to
find out whether there is a BIC in the perturbed structure with a slightly
different wavevector. BICs with a nonzero
wavevector are first studied by Porter and  Evans~\cite{porter05} and Hsu {\it et al.}~\cite{hsu13_2} for
a periodic array of rectangular cylinders and a photonic crystal slab
(a dielectric slab with a biperiodic array of circular air holes),
respectively. 
Their numerical results indicate that the BICs with a nonzero wavevector
exist continusouly with respect to strucural parameters, such as the
width or height of the rectangular cylinders, radius of the circular
air hole, thickness and the refractive index of the slab.  Zhen {\it et
  al.}~\cite{zhen14} suggested that these BICs are robust with respect to any small
perturbation that preserves the reflection symmetry in $z$ and the
inversion symmetry in the $xy$ plane.
In \cite{yuan17ol,yuan21rob}, we developed a formal 
theory for the robustness of BICs in structures with one and two
periodic directions, respectively. It was shown that only the generic
BICs satisfying a precise condition are robust. It turns out that the
non-generic BICs are precisely the super-BICs around which the
resonant modes have an ultra-high $Q$ factor~\cite{nan25prl}. In addition, it has
been shown that a non-generic BIC is a bifurcation point, when the
perturbation amplitude $\delta$ is regarded as a parameter~\cite{nan24ol}. In that
case, the perturbed structure can support different number of BICs for
$\delta > 0$ and $\delta < 0$, respectively~\cite{nan24opex}. 

In this paper, we follow our previous work~\cite{yuan17ol} and present a rigorous
theory on the robustness of generic BICs in two-dimensional (2D) structures
with a single periodic direction.  The theory developed in \cite{yuan17ol} is
rather formal. 
The BICs in the perturbed structures are constructed by
expanding the wave field, frequency and Bloch wavenumber in 
power series of $\delta$, but the
convergence of these power series was not established. In this work, we
formulate the problem in proper function spaces, establish the
solvability condition and estimate the norm for each term in the
power series, and prove the convergence of the power series
rigorously. The approach developed in
this paper can be extended to analyze the robustness of BICs in more
complicated structures, the parametric dependence of nonrobust BICs,
and the bifurcation of non-generic BICs.

\section{Bound states in the continuum and diffraction solutions}

\subsection{Basic equations}

We consider a 2D dielectric structure  that is translationally invariant
in $x$, periodic in $y$ with period $L$, and bounded in $z$ by
vacuum, where $\{x, y, z\}$ is a Cartesian coordinate  system. 
The dielectric function $\epsilon$  satisfies 
\begin{equation}
\label{eq:period_eps} \epsilon(y+L,z) = \epsilon({\bf r}), \quad
\forall  \ {\bf r} = (y,z) \in \mathbb{R}^2,
\end{equation}
and  $\epsilon({\bf r}) = 1$ for $|z| > d$. Moreover, $\epsilon \in
L^{\infty}(\mathbb{R}^2)$,  is real and positive. An example is
given in Fig.~\ref{fig:grating} 
\begin{figure}[h]
\centering
\includegraphics[scale=0.6]{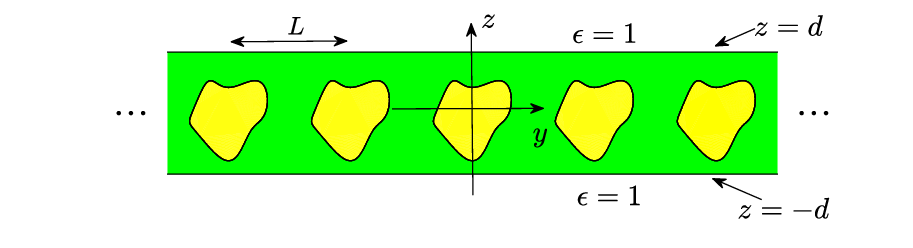}
\caption{An example of a 2D structure with 1D periodicity:a slab with a periodic array of air holes.}
\label{fig:grating}
\end{figure}
below. 

For 2D structures, assuming the wave fields are independent of $x$,
Maxwell's equations can  be reduced to two 
uncoupled scalar equations for transverse electric (TE) and transverse
magnetic (TM) waves, respectively. 
In this paper, we focus 
on the TE case, although all results can be easily extended to
the TM case. For the TE case,  the $x$-component of the electric field,
denoted as $u$, satisfies the following Helmholtz equation  
\begin{equation}
  \label{eq:helm}
\Delta u  + k^2 \epsilon({\bf r})\, u = 0, 
\end{equation}
where $\Delta = \partial_y^2 + \partial_z^2$, $k=\omega/c$ is the free-space wave number, $\omega$ is the angular frequency, and $c$ is the speed of light in vacuum.  Since the structure is periodic in $y$, we look for solutions of the form 
\begin{equation}
\label{eq:bloch_modes}  u({\bf r}) = \phi({\bf r}) e^{ \mathsf{i} \beta y},
\end{equation}
where $\phi$ is periodic in $y$ with the same period $L$,  and $\beta \in
(-\pi/L,  \pi/L]$ is the Bloch wave number. The
governing equation for $\phi$ is 
\begin{equation}
  \label{eq:helm_phi}
\Delta \phi  + 2 \mathsf{i} \beta \partial_y \phi  + [ k^2 \epsilon({\bf r}) -
\beta^2 ] \phi = 0.
\end{equation}

\subsection{Bound states in the continuum}

If for real values of $k$ and $\beta$, a nonzero solution of
\eqref{eq:helm} given in the Bloch form \eqref{eq:bloch_modes}
satisfies 
\begin{equation}
\label{eq:decay_condition} u({\bf r}) \to 0, \,  \mbox{as} \, z \to
\pm \infty, \, \forall y \in \mathbb{R}, 
\end{equation}
then the solution is a guided mode.
Guided modes with $0 < k < |\beta|$  are well known. They depend continuously on 
$\beta$ and $k$, and their existence in periodic dielectric structures
can be established~\cite{bonnet94}.

A BIC is a special guided mode with $k > |\beta|$. It is well-known
that periodic
structures with a reflection symmetry in the periodic direction, i.e.
\begin{equation}
\label{eq:yeven_eps} \epsilon({\bf r}) = \epsilon(-y,z), \quad \forall
{\bf r} \in \mathbb{R}^2,
\end{equation}
may support  symmetry-protected BICs which are antisymmetric in $y$
and have a zero $\beta$~\cite{bonnet94,evans94,shipman07}.  The existence of BICs that are not protected by symmetry is  more difficult to establish \cite{ches19, yu25, zhou25}.
BICs with a nonzero $\beta$ (the so-called propagating BICs) are first
reported in \cite{porter05,mari08,hsu13_2}.  It seems that if $L$ is the
true minimum period of $\epsilon({\bf r})$, the BICs appear as
isolated points in the $\beta$-$k$ plane. 
It is known that propagating BICs are much easier to find if the
periodic structure has reflection symmetries in both $y$ and $z$,
i.e.,  $\epsilon({\bf r})$ satisfies
condition \eqref{eq:yeven_eps} and  
\begin{equation}
\label{eq:zeven_eps} \epsilon({\bf r}) = \epsilon(y,-z), \quad \forall
{\bf r} \in \mathbb{R}^2. 
\end{equation}

Given the Bloch form  \eqref{eq:bloch_modes}, we can formulate the
eigenvalue problem for BICs as follows: 
Finding a nonzero function $u$ and a real pair $(\beta, k)$ satisfying
$\beta \in (-\pi/L, \pi/L]$, $ k > |\beta|$, and
\begin{eqnarray}
\label{eq:BICs}    \left\{ \begin{array}{ll} \Delta u  + k^2
                             \epsilon({\bf r})\,  u = 0,  &  {\bf r} \in \Omega, \\
                          u(L/2,z) = e^{ \mathsf{i} \beta L} u(-L/2,z), & z \in \mathbb{R},\\ 
                          \partial_y u (L/2,z) =  e^{ \mathsf{i} \beta L}
                             \partial_y u (-L/2,z),  & z \in \mathbb{R},\\
                          u({\bf r})  \to 0 \quad \mbox{as} \quad   z
                             \to  \pm \infty, &| y | < L/2,
                          \end{array} \right.
\end{eqnarray}
where $ \Omega =  \left\{ (y,z) \in \mathbb{R}^2 :  |y| < L/2, |z| <
  \infty \right\}$ is the domain for one period of the structure. The
boundary conditions of $u$ in the $y$ 
direction are the quasi-periodic conditions. In terms of $\phi$, the
eigenvalue problem for BICs is: Finding a nonzero function $\phi$ and
a real pair $(\beta, k)$ satisfying $\beta \in (-\pi/L, \pi/L]$,  $k
> |\beta|$, and
\begin{eqnarray}
  \label{eq:BICs_phi}
  \left\{ \begin{array}{ll}
                                   \Delta \phi + 2 \mathsf{ i } \beta \partial_y
            \phi  + [  k^2 \epsilon({\bf r}) - \beta^2 ]  \phi = 0,  &  {\bf r} \in \Omega, \\
    \phi(L/2,z) = \phi(-L/2,z), &  z  \in \mathbb{R},\\ 
\partial_y \phi (L/2,z) =  \partial_y  \phi (-L/2,z),  & z \in \mathbb{R},\\
                          \phi \to 0 \quad \mbox{as} \quad  z \to  \pm \infty, & | y | < L/2.
                          \end{array} \right.
\end{eqnarray}

\subsection{Simple properties of BICs}
\label{sec:propertiesBIC}
Let $u_*({\bf r})$ be a BIC with free-space wave number $k_*$ and
Bloch wave number $\beta_*$  in a periodic structure with dielectric
function $\epsilon_*({\bf r})$ satisfying the symmetry conditions 
\eqref{eq:yeven_eps} and \eqref{eq:zeven_eps}. If the 
BIC $u_*$ is non-degenerate, i.e., the eigenspace of problem
\eqref{eq:BICs} is one-dimensional, then we can scale $u_*$ 
such that
\begin{equation}
\label{eq:PTsym_u} u_*({\bf r}) =  \overline{u}_*(-y,z) , \quad 
\forall {\bf r} \in \mathbb{R}^2, 
\end{equation}
i.e.,  it is $\mathcal{PT}$-symmetric \cite{porter05}. Here, $\overline{u}_*$ denotes
the complex conjugate of $u_*$.   In addition, since the structure has reflection symmetry in $z$, the
BIC $u_*$ must be either even or odd in $z$, i.e., 
\begin{equation*}
\label{eq:zeven_u} u_*({\bf r}) = u_*(y,-z), \quad \mbox{or} \quad u_*({\bf r}) = -u_*(y,-z),
\end{equation*}
for $ {\bf r} \in \mathbb{R}^2$.

In the homogeneous medium (assumed to be air)  for $z > d$ and $z<-d$, the BIC $u_*$ can be
expanded in Fourier series
\[
  u_*({\bf r}) = \sum\limits_{m = -\infty}^{\infty} c_{m}^\pm
  e^{ \mathsf{ i } ( \beta_m y \pm \alpha_m z)}, \quad \pm z > d, 
\]  
where $c^{\pm}_m$ are coefficients, $\beta_m = \beta_* + 2 m \pi /L$,  and 
\begin{equation}
\label{eq:alpha}
\alpha_m =  \left\{ \begin{array}{ll}  \sqrt{k^2_* - \beta_m^2}, & k^2_* \geq \beta_m^2, \\
                                                           \mathsf{i} \sqrt{ \beta_m^2 - k^2_*}, & k^2_* < \beta_m^2.
                                                           \end{array}  \right. 
\end{equation}
To focus on the simplest BICs, we further assume that 
\begin{equation}
\label{eq:cond_k_beta} |\beta_*| < k_* < 2 \pi/L - |\beta_*|. 
\end{equation}
In that case, only $\alpha_0=\sqrt{k_*^2 -\beta_*^2}$ is real, all
other $\alpha_m$ for $m \neq 0$  are purely imaginary.  For $m\ne 0$,
we let 
\begin{equation}
\gamma_m = - \mathsf{i} \alpha_m =  \sqrt{\beta_m - k^2} > 0. 
\end{equation}
If condition  \eqref{eq:cond_k_beta} is satisfied, then 
$\gamma_m \geq \gamma_{\dagger} =  \mbox{min}(\gamma_{-1}, \gamma_1)$ for all $m \neq 0$. 

Since the BIC must decay to zero as $z\to \pm \infty$,
we have $c_0^\pm = 0$, and thus 
\begin{equation}
\label{eq:BIC_fourier} u_*({\bf r})  = \sum_{ m \neq 0}  c^{\pm}_m e^{
  \mathsf{ i } \beta_m y \mp \gamma_m z }, \quad \pm z > d. 
\end{equation}

\subsection{Diffraction solutions}
\label{sec:diffraction_solution}

For periodic structures with a dielectric function given in
section 2.1, we can 
formulate diffraction problems by specifying plane incident waves in
the surrounding homogeneous media. Mathematical
theories of diffraction problems were
established several decades ago~\cite{bonnet94,bao94}.
It is well-known that a diffraction problem  has a unique solution for all but a sequence of countable frequencies.
If there is a BIC, the corresponding   diffraction problem 
at the frequency and Bloch wavenumber of  the
BIC has nonunique solutions~\cite{bonnet94,shipman07}. 

If there is a nondegenerate BIC $\{ u_*, \beta_*, k_*\}$ in a periodic structure
with $\epsilon_*$ satisfying \eqref{eq:yeven_eps} and
\eqref{eq:zeven_eps},
and the BIC is even in $z$ and satisfies condition
\eqref{eq:cond_k_beta},  we consider the following incident wave
\begin{equation}
  \label{eq:incidentwave}
  u_*^{(\text{in})}({\bf r}) =  \left\{ \begin{array}{ll} 
                                                    e^{ \mathsf{ i } ( \beta_* y -  \alpha_0 z)}, & z > d, \\
                                                    e^{ \mathsf{ i } (\beta_* y + \alpha_0 z)}, & z < - d,
                                                    \end{array} \right.
\end{equation}
and seek  a diffraction solution $v({\bf r})$ satisfying                                                 
\begin{eqnarray}
\label{eq:difffraction}    \left\{ \begin{array}{ll} \Delta v  + k^2_*
                                     \epsilon_*({\bf r}) v = 0,  & {\bf r} \in \Omega, \\
                          v(L/2,z) = e^{ \mathsf{ i } \beta_* L} v(-L/2,z), & z \in \mathbb{R},\\ 
                         \partial_y v (L/2,z) =  e^{ \mathsf{ i } \beta_* L} \partial_y v (-L/2,z),  & z \in \mathbb{R},\\
                          v - u_*^{(\text{in})} \mbox{ is outgoing as }     |z| \to  \infty, &| y | < L/2.
                          \end{array} \right.
\end{eqnarray}
Since $u_*^{(\text{in})}$ and $\epsilon_*$ are both even in $z$, we assume $v$ is
also even in $z$. The solution can be written as 
\begin{equation*}
v({\bf r}) = e^{ \mathsf{ i } (  \beta_* y \mp \alpha_0 z )} + \rho e^{ \mathsf{ i } (  \beta_*
  y \pm \alpha_0 z )} +  \sum_{m \neq 0} a_m e^{  \mathsf{ i } \beta_m y \mp \gamma_m z}, \quad \pm  z > d, 
\end{equation*}
where $a_m$ ($m\ne 0$) are coefficients and $|\rho| = 1$ due to energy conservation. Let $\rho = e^{2 \mathsf{i}
  \theta}$ for a real $\theta$, then $v_* = e^{-\mathsf{i} \theta} v$ is a
diffraction solution for incident wave $e^{-\mathsf{i} \theta} u_*^{(\text{in})}$. It can be expanded as
\begin{equation}
\label{eq:v_fourier}
v_*({\bf r}) = e^{- \mathsf{i} \theta} e^{ \mathsf{ i } (  \beta_* y \mp \alpha_0 z )} + e^{ \mathsf{i} \theta} e^{ \mathsf{ i } (  \beta_*
  y \pm \alpha_0 z )} +  e^{- \mathsf{i} \theta} \sum_{m \neq 0} a_m e^{  \mathsf{ i } \beta_m y \mp \gamma_m z}, \quad \pm  z > d.
\end{equation}
It is easy to verify that $ \overline{v}_*(-y,z)$ is also a 
diffraction solution for the same incident wave. Therefore, we can assume $v_* $
satisfies the following $\mathcal{PT}$-symmetry  condition
\begin{equation}
\label{eq:PTsym_v} v_*({\bf r}) =  \overline{v}_*(-y,z) , \quad 
\forall {\bf r} \in \mathbb{R}^2.
\end{equation}
Since the BIC is nondegenerate,   $v_* + \eta u_*$ is also a
solution for any constant $\eta$. We can assume that the diffraction solution $v_*$ is orthogonal to the BIC, that is,
\begin{equation}
    \label{eq:orth_v} \int_{\Omega} \overline{v}_* \epsilon_* u_* d \mathbf{r} = 0.
\end{equation}

Similarly, by choosing incident plane waves with opposite signs for
$z>d$ and $z<-d$, we can obtain a diffraction solution
that is odd in $z$ and $\mathcal{PT}$-symmetric in $y$.


\section{Inhomogeneous Helmholtz equation: solvability}
\label{sec:solvability}
In this section, we analyze an inhomogeneous Helmholtz equation with the
same frequency and Bloch wave number as the BIC. As in the previous
section, we assume the periodic structure is described by the
dielectric function $\epsilon_*$ 
 satisfying \eqref{eq:period_eps},
\eqref{eq:yeven_eps} and \eqref{eq:zeven_eps}, and that the BIC $\{ u_*,
\beta_*, k_*\}$ is nondegenerate, even in $z$, and satisfies
\eqref{eq:cond_k_beta}. Since the structure and the BIC are both even
in $z$, we restrict our attention to inhomogeneous equations for which
the right-hand side $f({\bf r})$ 
and the solution $w({\bf r})$  are
also even in $z$. Consequently, the problem can be formulated on 
$$ \Omega^+ = \left\{ (y,z) \in \mathbb{R}^2 :  |y| < L/2, z > 0
\right\}. $$

The inhomogeneous boundary value problem (BVP) is 
\begin{eqnarray}
\label{eq:BVP1}    \left\{ \begin{array}{ll} \Delta w + k_*^2
                             \epsilon_*({\bf r})  w =  f({\bf r}),  & {\bf r} \in \Omega^+, \\
                          w(L/2,z) = e^{ \mathsf{ i } \beta_* L} w(-L/2,z), &  z > 0,\\ 
\partial_y w (L/2,z) =  e^{ \mathsf{ i } \beta_* L}\partial_y w(-L/2,z),  & z > 0,\\
\partial_z w(y,0) = 0, & |y| < L/2, \\
                          w  \quad \mbox{is outgoing as} \quad  z \to + \infty, & | y | < L/2. 
                          \end{array} \right. 
\end{eqnarray}
In particular, we assume that 
\begin{equation}
\label{eq:f_cond1} f({\bf r}) =  \sum_{m\neq 0}  P_{m}(z) e^{ \mathsf{ i } \beta_m 
  y - \gamma_m z} , \quad z > d, 
\end{equation}
where $P_m(z)$ are polynomials in $z$. The objective of this section is two-fold. First,  we
 show that BVP \eqref{eq:BVP1} has a solution
 if  $f$  satisfies
 \begin{eqnarray}
&& \label{eq:f_cond2} \int_{\Omega^+} \overline{u}_*({\bf r})  f({\bf r})  d {\bf r} = 0.
 \end{eqnarray}
Second, we show that if $f$ also satisfies
\begin{equation} \label{eq:f_cond3}
\int_{\Omega^+} \overline{v}_*({\bf r})  f({\bf r}) d {\bf r} = 0,
\end{equation}
then the solution $w$ of BVP \eqref{eq:BVP1}  satisfies
$w \to 0$ as $z \to +\infty$.

Note that the dielectric function $\epsilon_* = 1$  for $z > d$, if BVP \eqref{eq:BVP1} has a solution $w$, then it can be expanded as
\begin{equation}
\label{eq:w_fourier2} w(y,z) = q_0 e^{ \mathsf{ i } (\beta_* y + \alpha_0 z)} + \sum\limits_{ m\neq 0} Q_m(z) e^{  \mathsf{ i } \beta_m y - \gamma_m z }, \quad z > d,
\end{equation}
where $q_0$ is a constant, and for $m\neq 0$, $Q_m(z)$  satisfies
\begin{eqnarray}
\label{eq:Q_m} \left\{ \begin{array}{ll}
Q''_m(z) - 2 \gamma_m Q'_m(z) = P_m(z), & z > d, \\
Q_m(z) e^{- \gamma_m z} \to 0, &  z \to +\infty.
\end{array}
 \right.
\end{eqnarray}
The condition that $w \to 0$ as $z \to +\infty$ is equivalent to $q_0 = 0$.

\subsection{Spaces}
As in \cite{bonnet94}, the following domains will be used.
\begin{itemize}
\item[(1)] $\Omega_d = \left\{ (y,z) \in \mathbb{R}^2 : | y | < L/2, 0 < z < d \right\}$

\item[(2)] $\Gamma_d = \left\{ (y,z) \in \mathbb{R}^2  : | y | < L/2, z = d \right\}, $

 \item[(3)] $\Gamma_0 = \left\{ (y,z) \in \mathbb{R}^2  :  | y | < L/2, z  = 0 \right\},$
 
 \item[(4)] $\Omega_d^c = \Omega^+ \backslash \overline{\Omega}_d = \left\{ (y,z) \in \Omega^+ : (y,z) \notin  \overline{\Omega}_d \right\}$ ,  where $\overline{\Omega}_d$ is the closure of $\Omega_d$,
\item[(5)] $ \Gamma =  \left\{ y \in \mathbb{R}  :  | y | < L/2 \right\}. $
\end{itemize}
The following spaces will be used.  These spaces have been studied in \cite{alber79,wilcox84}.
\begin{itemize}
\item[(1)] $C^{\infty}_{\beta}(\mathbb{R}^2)$ is the set of all functions which are $C^{\infty}$ on $\mathbb{R}^2$, satisfy the qusi-periodic condition in $y$ and vanish for large $z$.

\item[(2)] $C^{\infty}_{\beta}(\Omega_d)$ is the set of the restrictions to $\Omega_d$ of all functions of  $C^{\infty}_{\beta}(\mathbb{R}^2)$.

\item[(3)] $ H^1_{\beta}(\Omega_d)$ is the smallest closed subspace of $H^{1}(\Omega_d)$ which contains  $C^{\infty}_{\beta}(\Omega_d)$.

\item[(4)] $ H^1_{\beta, 0}(\Omega_d)$ is the space of all functions  $ w \in H^{1}_{\beta}(\Omega_d)$ such that $\partial_z w = 0$ on $\Gamma_0$.

\item[(5)] $H^{1/2}_{\beta}(\Gamma) :=  \left\{ v = \sum\limits_{m=-\infty}^{\infty} v_m e^{ \mathsf{ i } \beta_m y} : \sum\limits_{m=-\infty}^{\infty} \left( 1 + \beta^2_m \right)^{1/2} |v_m|^2 < \infty  \right\} ,$ with norm
$$ \| v \|^2_{H^{1/2}_{\beta}(\Gamma)} =  L \sum\limits_{m=-\infty}^{\infty} \left( 1 + \beta^2_m \right)^{1/2} |v_m|^2.$$
This norm is equivalent to the classical $H^{1/2}$-norm on $H^{1/2}_{\beta}(\Gamma) $. 

\item[(6)] $H^{-1/2}_{\beta}(\Gamma): = \left\{ s = \sum\limits_{m=-\infty}^{\infty} s_m e^{ \mathsf{ i } \beta_m y} : \sum\limits_{m=-\infty}^{\infty} \left( 1 + \beta^2_m \right)^{-1/2} |s_m|^2 < \infty  \right\}$ is the dual space of $H^{1/2}_{\beta}(\Gamma) $. The norm is
$$ \| s \|^2_{H^{-1/2}_{\beta}(\Gamma)} =  L \sum\limits_{m=-\infty}^{\infty} \left( 1 + \beta^2_m \right)^{-1/2} |s_m|^2.$$
The duality product between $H^{1/2}_{\beta}(\Gamma)$ and $H^{-1/2}_{\beta}(\Gamma)$ is defined as
\begin{equation}
\label{eq:duality_pairing} \left< s,v \right>  = L \sum_{m=-\infty}^{\infty} s_m \overline{v}_m,
 \end{equation}
for any $s \in H^{-1/2}_{\beta}(\Gamma)$ and $v \in H^{1/2}_{\beta}(\Gamma)$.

\item[(7)]  Throughout this paper, we use $(\cdot, \cdot)_X$ to denote the standard inner product of the Hilbert space $X$, and define $\| u \|_X = \sqrt{(u, u)_X}$ for any $u \in X$.
\end{itemize}

\subsection{Truncation of the domain}
\label{sec:truncation}

For each $v \in H^{1/2}_{\beta}(\Gamma)$, define the operator $\mathcal{T} : H^{1/2}_{\beta}(\Gamma) \to H^{-1/2}_{\beta}(\Gamma)$  by
\begin{equation}
\label{eq:T_operator}\mathcal{T} v = \sum_{m=-\infty}^{\infty} \mathsf{i} \alpha_m v_m e^{ \mathsf{ i } \beta_m y} ,
\end{equation}
where $\{ v_m \}_{m=-\infty}^{\infty}$ are the Fourier coefficients of $v$.
Note that the operator $\mathcal{T}$ is linear. Using the operator $\mathcal{T}$ and the expression \eqref{eq:w_fourier2}, we obtain the following boundary condition for $w$ at $z = d$:
\begin{equation}
\label{eq:BC_w_D} \left. \frac{\partial w}{\partial z} \right|_{(y,d)} = \mathcal{T} w(y,d) + h(y), 
 \end{equation}
 where 
 \begin{equation}
\label{eq:h} h(y) = \sum\limits_{ m\neq 0} Q'_m(d) e^{  \mathsf{ i } \beta_m y - \gamma_m d}.
\end{equation}
Note that the function $h(y)$ depends only on  $f$. If $f \equiv 0$, then $h \equiv 0$.

With the boundary condition \eqref{eq:BC_w_D}, we obtain the following BVP  on the finite domain $\Omega_d $: Finding $w$ such that
 \begin{eqnarray}
\label{eq:BVP2}    \left\{ \begin{array}{ll} \Delta w + k_*^2 \epsilon_* w =  f,  & (y,z) \in \Omega_d, \\
                          w(L/2,z) = e^{ \mathsf{ i } \beta_* L} w(-L/2,z), &  0 < z <d ,\\ 
                           {\partial_y w}(L/2,z) =  e^{ \mathsf{ i } \beta_* L} {\partial_y w}(-L/2,z),  & 0 < z < d,\\
                            {\partial_z w} = 0, & (y,z) \in \Gamma_0, \\
                          {\partial_z w} = \mathcal{T} w + h(y), & (y,z) \in \Gamma_d.
                          \end{array} \right.
\end{eqnarray}

It is easy to check that the restriction of BIC $u_*$ to $\Omega_d$ is a non-trivial solution of BVP \eqref{eq:BVP2} with $f \equiv 0$. Therefore,  BVP \eqref{eq:BVP2} is singular. 

We show that BVPs \eqref{eq:BVP1}  and  \eqref{eq:BVP2} are equivalent.
\begin{theorem}
\label{theorem:equivalence} Let $f \in L^2(\Omega^+)$ satisfy \eqref{eq:f_cond1}, and $h$ be defined by \eqref{eq:h}. If $w$ is a solution of  BVP \eqref{eq:BVP1}, then its restriction to $\Omega_d$, denoted $w|_{\Omega_d}$, is a solution of  BVP \eqref{eq:BVP2}. Conversely, if $w$ is a solution of  BVP \eqref{eq:BVP2}, then it can be extended to a solution of BVP \eqref{eq:BVP1}.
\end{theorem}

\begin{proof}
The first part follows immediately from the derivation of the boundary condition \eqref{eq:BC_w_D}.  Conversely, assume that $\tilde{w}$ is a solution of BVP \eqref{eq:BVP2},  and define
\begin{equation}
\label{eq:w_extend}
{w} = \left\{ \begin{array}{ll} 
\tilde{w}, & (y,z) \in \Omega_d \\
q_0 e^{ \mathsf{ i } (\beta_* y + \alpha_0 z)} + \displaystyle \sum\limits_{m\neq 0} {Q}_m(z) e^{  \mathsf{ i } \beta_m y - \gamma_m z}, & (y,z) \in \Omega_d^c,
\end{array} \right.
\end{equation}
where $q_0 = \frac{e^{-i \alpha_0 d}}{L} \int_{-L/2}^{L/2} \tilde{w}(y,d) e^{-\mathsf{ i } \beta_* y} dy$, and  $Q_m(z)$ for $m\neq 0$ are the polynomials that satisfy \eqref{eq:Q_m}, and 
$$ Q_m(d) = \frac{e^{\gamma_m d}}{L} \int_{-L/2}^{L/2} \tilde{w}(y,d) e^{-\mathsf{ i } \beta_m y} dy.  $$
With the definition of $h(y)$ in \eqref{eq:h}, one can easily verify that $w$  and $\partial_z w$ are continuous at $z=d$. By the definition of $Q_m(z)$,  $w$  satisfies the inhomogeneous Helmholtz equation of BVP \eqref{eq:BVP1} in $\Omega_d^c$, and it is outgoing as $z \to + \infty$. Thus, $w$ is a solution of BVP \eqref{eq:BVP1}. 
\end{proof} 

In what follows, unless otherwise noted, the same notation will be used for the solutions of BVPs \eqref{eq:BVP1} and \eqref{eq:BVP2}.

\subsection{Existence}

We now prove that there exist solutions to BVPs \eqref{eq:BVP2} and \eqref{eq:BVP1} if $f$ satisfies \eqref{eq:f_cond2}  using a  method similar to that in \cite{bonnet94}. BVP \eqref{eq:BVP2}  admits the following variational formulation:
\begin{equation}
\label{eq:variation} \mbox{Find} \, \, w \in H^{1}_{\beta, 0}(\Omega_d) \, \, \mbox{such that} \, \,  \mathbf{a}(w,v) = \mathbf{l}(v), \, \, \, \,  \forall v \in H^1_{\beta, 0}(\Omega_d),
\end{equation}
where the forms $\mathbf{a}(w,v)$ and $\mathbf{l}(v)$ are defined, respectively, by
\begin{eqnarray}
\mathbf{a}(w,v) &=& \int_{\Omega_d} \left( \nabla w \cdot \nabla \overline{v} - k_*^2 \epsilon_* w \overline{v} \right) d {\bf r} - \left< \mathcal{T} w(y,d), v(y,d) \right>, \\
\mathbf{l}(v) &=&\int_{-L/2}^{L/2} h(y) \overline{v}(y,d) dy - \int_{\Omega_d} f \overline{v} d {\bf r},
\end{eqnarray}
and $\left< \cdot, \cdot \right>$ is defined by \eqref{eq:duality_pairing}. It can be readily checked that
$$ \left< \mathcal{T} w(y,d), v(y,d) \right> =  \int_{-L/2}^{L/2} \overline{v}(y,d) \mathcal{T} w(y,d) d y.$$

We decompose $\mathbf{a}(w,v)$ as
$$ \mathbf{a}(w,v) = \mathbf{b}(w,v) + \mathbf{c}(w,v), $$
where
\begin{equation*}
\label{eq:linear_b} \mathbf{b}(w,v) =  \int_{\Omega_d} \left[ \nabla w \cdot \nabla \overline{v} + \left(1 + k_*^2 \| \epsilon_*\|_{L^{\infty}(\Omega_d)} - k_*^2 \epsilon_* \right) w \overline{v} \right] d {\bf r} - \left<\mathcal{T} w(y,d), v(y,d) \right>, 
 \end{equation*}
$$  \mathbf{c}(w,v) = - \left(k_*^2 \| \epsilon_*\|_{L^{\infty}(\Omega_d)} + 1 \right)  \int_{\Omega_d} w \overline{v} d {\bf r}.$$

Sesquilinear forms $\mathbf{b}(w,v)$  and  $\mathbf{c}(w,v)$ are continuous and bounded. We can define two bounded operators $\mathbf{B}$ and $\mathbf{C}$ of $H^1_{\beta, 0}(\Omega_d)$ such that
$$ (\mathbf{B}w, v)_{H^1_{\beta,0}(\Omega_d)} = \mathbf{b}(w,v), \quad  (\mathbf{C}w, v)_{H^1_{\beta,0}(\Omega_d)} = \mathbf{c}(w,v), \quad   \forall u, v \in H^1_{\beta, 0}(\Omega_d).$$
Using a method similar to that in Lemma 3.1 of \cite{bonnet94}, it can be shown that operator $\mathbf{B}$ is invertible and operator $\mathbf{C}$ is compact on $H^1_{\beta,0}(\Omega_d)$. 

Let $w_{\dagger}$ denotes the unique element of $H^1_{\beta,0}(\Omega_d)$ such that 
$$ (w_{\dagger}, v)_{H^1_{\beta,0}(\Omega_d)}  = \mathbf{l}(v), \quad \forall v \in H^1_{\beta,0}(\Omega_d).$$
Then BVP \eqref{eq:BVP2} can be reformulated as follows:
\begin{equation}
\label{eq:operator}\mbox{Find} \, \, w \in H^1_{\beta,0}(\Omega_d) \, \,  \mbox{such that} \, \,  \mathbf{B} w + \mathbf{C}w = w_{\dagger}.
\end{equation}
The existence of  a solution to problem \eqref{eq:operator} is related to the existence of  a solution to the homogeneous problem
\begin{equation}
\label{eq:operator_homo}\mbox{Find} \, \, w \in H^1_{\beta,0}(\Omega_d) \, \,  \mbox{such that} \, \,  \mathbf{B} w + \mathbf{C}w = 0,
\end{equation}

Since BIC $u_*$ is a solution of problem \eqref{eq:operator_homo} and it  is nondegenerate, the solution space of problem \eqref{eq:operator_homo} is one-dimensional. By Fredholm alternative,  problem \eqref{eq:operator} [i.e., problem \eqref{eq:variation}, or equivalently, BVP \eqref{eq:BVP2}] has a solution if and only if $(w_{\dagger}, u_*)_{H^1_{\beta,0}(\Omega_d)} = \mathbf{l}(u_*) = 0$. We show that this condition holds if  $f$ satisfies \eqref{eq:f_cond2}. 

\begin{theorem}
\label{theorem:existenceBVP2} Let  $f \in L^2(\Omega^+)$ satisfy \eqref{eq:f_cond1} and \eqref{eq:f_cond2}, and  $h$ be defined by \eqref{eq:h}. Then there exists a solution to BVPs \eqref{eq:BVP2} and \eqref{eq:BVP1}.
\end{theorem}

\begin{proof} 
We only need to show 
$$\mathbf{l}(u_*) = \int_{\Omega_d} \overline{u}_* f d {\bf r} -\int_{-L/2}^{L/2} \overline{u}_*(y,d) h(y) dy = 0.$$
Using \eqref{eq:BIC_fourier} and \eqref{eq:f_cond1}, we have
\begin{eqnarray*} 
\int_{\Omega_d^c} \overline{u}_* f d {\bf r} = \int_{d}^{+\infty} dz \int_{-L/2}^{L/2}   \overline{u}_* f dy  =
L \sum\limits_{m\neq 0} \overline{c}^+_m \int_{D}^{+\infty} P_m(z)  e^{-2 \gamma_m z } dz.
\end{eqnarray*}
Using \eqref{eq:Q_m} to eliminate $P_m(z)$, and applying integration by parts,  we have
 \begin{eqnarray*} 
\int_{\Omega_d^c} \overline{u}_* f d {\bf r} = - L \sum\limits_{m\neq 0} \overline{c}^+_m Q'_m(d) e^{-2 \gamma_m d}.
\end{eqnarray*}
On the other hand, it is easy to verify that
$$ \int_{-L/2}^{L/2} \overline{u}_*(y,d) h(y) dy = L \sum\limits_{m\neq 0} \overline{c}^+_m Q'_m(d) e^{-2 \gamma_m d}. $$
Comparing the right-hand sides of the above two equations, we have
$$ \int_{-L/2}^{L/2} \overline{u}_*(y,d) h(y) dy = - \int_{\Omega_d^c} \overline{u}_* f d {\bf r}.$$
Therefore, 
$$ \int_{\Omega_d} \overline{u}_* f d {\bf r} - \int_{-L/2}^{L/2} \overline{u}_*(y,d) h(y) dy = \int_{\Omega_d} \overline{u}_* f d {\bf r}  + 
\int_{\Omega_d^c} \overline{u}_* f d {\bf r}  = \int_{\Omega^+} \overline{u}_* f d {\bf r} = 0.   $$ 
The last equation is due to \eqref{eq:f_cond2}.  By Fredholm alternative,  problem \eqref{eq:operator} [i.e.,  BVP \eqref{eq:BVP2}] has a solution. By Theorem \ref{theorem:equivalence}, this solution can be extended to a solution of BVP \eqref{eq:BVP1}.
\end{proof}

\noindent {\bf Remark }: Due to the existence of BIC $u_*$, solutions to BVPs \eqref{eq:BVP1} and \eqref{eq:BVP2} are not unique. More specifically, if $w$ is a solution, then $w + \eta u_*$ is also a solution for any constant $\eta$. On the other hand, since we assume that the BIC is nondegenerate, all solutions can be expressed as $w + \eta u_*$.

\subsection{Decaying property}
Next, we show that if $f$ also satisfies \eqref{eq:f_cond3}, then solutions of BVP \eqref{eq:BVP1} decay to zero as $z \to +\infty$. 
To this end, we need the following lemma, which can be proved by a procedure similar to that used in Theorem \ref{theorem:existenceBVP2}.
\begin{lemma}
\label{lemma:f_cond3} 
Let $f \in L^2(\Omega^+)$  satisfy  \eqref{eq:f_cond1}, and $h$ be defined by \eqref{eq:h}. Then
$$\int_{\Omega_d} \overline{v}_* f d {\bf r} - \int_{-L/2}^{L/2} \overline{v}_*(y,d) h(y) dy = 0.$$
\end{lemma}

We are now in a position to prove the existence of a decaying solution.
\begin{theorem}
\label{theorem:decay} Let $f \in L^2(\Omega^+)$ satisfy \eqref{eq:f_cond1} -- \eqref{eq:f_cond3}. Then BVP \eqref{eq:BVP1} admits a solution satisfying $w \to 0$ as $z \to +\infty$.
\end{theorem}

\begin{proof}
By Theorem \ref{theorem:existenceBVP2}, conditions \eqref{eq:f_cond1} and \eqref{eq:f_cond2} ensure that BVP \eqref{eq:BVP2} has a solution $w$. Multiplying the equation of $w$ by $\overline{v}_*$ (the diffraction solution defined in Section~\ref{sec:diffraction_solution}),  integrating over the domain $\Omega_d$, and applying the integration by parts, we obtain
\begin{equation}
\label{eq:tmp_lemma_decay} \int_{\partial \Omega_d} \left( \overline{v}_* \frac{\partial w}{\partial \nu}  - w \frac{\partial \overline{v}_*}{\partial \nu} \right) d s = \int_{\Omega_d } \overline{v}_* f  d {\bf r}, 
 \end{equation}
where $\partial \Omega_d$ is the boundary of $\Omega_d$ and $\nu$ is the outward unit normal vector. In the above derivation, we have used the equation satisfied by $v_*$. Applying the quasi-periodic condition in the $y$-direction and the zero Neumann condition at $z=0$ to both $w$ and $v_*$,  we find that \eqref{eq:tmp_lemma_decay}  becomes 
$$\int_{-L/2}^{L/2} \left[ \overline{v}_*(y,d) \frac{\partial w}{\partial z}(y,d) - w(y,d) \frac{\partial \overline{v}_*}{\partial z} (y,d) \right] dy  = \int_{\Omega_d } \overline{v}_* f  d {\bf r}.$$
Using the boundary condition of $w$ at $z =d$ and the  expansion \eqref{eq:v_fourier} of $v_*$ for $z > d$, we have
$$  2 \mathsf{i} \alpha_0  e^{- \mathsf{i} (\theta + \alpha d)}   \int_{-L/2}^{L/2}  w(y,d) e^{-\mathsf{ i } \beta_* y} dy = \int_{\Omega_d } \overline{v}_* f  d {\bf r} - \int_{-L/2}^{L/2} \overline{v}_*(y,d) h(y) dy.$$
According to Lemma \ref{lemma:f_cond3}, the right-hand side of this equation is zero. Thus, 
$$\int_{-L/2}^{L/2}  w(y,d) e^{-\mathsf{ i } \beta_* y} dy = 0,$$ 
i.e., the zeroth-order Fourier coefficients of $w(y,d)$ is zero. By Theorem \ref{theorem:equivalence}, when extending $w$ to a solution (still denoted as $w$) of BVP \eqref{eq:BVP1} by \eqref{eq:w_extend}, we have $q_0 = 0$, and thus $w \to 0$ as $z \to +\infty$.
\end{proof}

\section{Inhomogeneous Helmholtz equation: boundness}
In this section, we  show that the solution of BVP \eqref{eq:BVP1},  orthogonal to the BIC $u_*$, can be bounded by the inhomogeneous term $f$.  First, we introduce the following lemma.

\begin{lemma}
\label{lemma:energy}  Let  $f \in L^2(\Omega^+)$ satisfy ~\eqref{eq:f_cond1}  and \eqref{eq:f_cond2}, and $h$ be defined by \eqref{eq:h}.  Then, there exists a constant $C$  such that
$$  \| w \|_{H^1(\Omega_d)} \leq C \left[ \| w \|_{L^2(\Omega_d)}  + \| f \|_{L^2(\Omega_d)} + \| h \|_{L^2(\Gamma)}  \right] , $$
where $w $ is a solution of BVP \eqref{eq:BVP2}, and $C$ is independent of $f, h$ and $w$.
\end{lemma}

\begin{proof} 
Multiplying the equation of $w$ in BVP \eqref{eq:BVP2} by $\overline{w}$,  integrating over the domain $\Omega_d$, and applying the integration by parts, we obtain
\begin{eqnarray} 
\label{eq_lemma_energy2}
  && \int_{\Omega_d} |\nabla w|^2 d {\bf r} - \int_{-L/2}^{L/2} \overline{w}(y,d) \mathcal{T} w(y,d) dy     \\ 
&&= k^2_* \int_{\Omega_d} \epsilon_* |w|^2 d {\bf r} + \int_{-L/2}^{L/2} h(y) \overline{w}(y,d) dy - \int_{\Omega_d} f \overline{w} d {\bf r}.  \nonumber 
\end{eqnarray}
By Theorem {\ref{theorem:decay}},  the zeroth-order Fourier coefficient of $w(y,d)$ being zero. Thus,
$$ \int_{-L/2}^{L/2} \overline{w}(y,d) \mathcal{T} w(y,d) dy = - L\sum\limits_{m\neq 0} \gamma_m |w_m|^2 \leq 0,  $$
where $w_m $ are the Fourier coefficients of $w(y,d)$.
Using the above relation to \eqref{eq_lemma_energy2}, we have
\begin{eqnarray} 
\label{eq_lemma_energy1} \int_{\Omega_d} |\nabla w|^2 d {\bf r}     
& \leq & \left( k_*^2 \|\epsilon_* \|_{L^{\infty}(\Omega_d)} + \frac{1}{2} \right) \int_{\Omega_d}  |w|^2 d {\bf r}   \\
&& + \int_{-L/2}^{L/2} | h(y) \overline{w}(y,d) | dy + \frac{1}{2} \int_{\Omega_d} | f |^2d {\bf r}.  \nonumber 
\end{eqnarray}
Applying Cauchy's inequality with any $\eps >0$ to the second term on the right-hand side, we get
\begin{equation}
\label{eq:hw}
\int_{-L/2}^{L/2} | h(y) \overline{w} (y,d)| dy  \leq \eps  \int_{-L/2}^{L/2}  |w(y,d) |^2 dy  + \frac{1}{4 \eps}  \int_{-L/2}^{L/2}  |h |^2 dy. 
\end{equation}
Using the definition of the norm in space $H^{1/2}_{\beta}(\Gamma)$, we have
$$  \int_{-L/2}^{L/2}  |w(y,d) |^2 dy \leq \| w(y,d) \|^2_{H^{1/2}_{\beta}(\Gamma)} \leq C_1 \| w \|^2_{H^1(\Omega_d)},$$
where $C_1 >0$ is a constant. The second inequality arises from the fact  that  $H^{1/2}_{\beta}(\Gamma)$ is the space of the traces on $\Gamma_d$ of all functions of $H^1_{\beta,0}(\Omega_d)$, and the trace operator is bounded in $H^1_{\beta,0}(\Omega_d)$ \cite{bonnet94,leoni}. Substituting the above inequality into \eqref{eq:hw}, we have
$$  \int_{-L/2}^{L/2} | h(y) \overline{w}(y,d) | dy \leq \eps  C_1 \| \nabla  w \|^2_{L^2(\Omega_d)}  + \eps C_1  \| w \|^2_{L^2(\Omega_d)}  + \frac{1}{4 \eps}  \| h \|^2_{L^2(\Gamma)}. $$
Inserting it into \eqref{eq_lemma_energy1} and choosing $\eps$ such that $\eps C_1 \leq 1/2$, we have
$$  \frac{1}{2}   \|\nabla w \|^2_{L^2(\Omega_d)}  \leq C_2 \| w \|^2_{L^2(\Omega_d)} + \frac{1}{4 \eps}  \| h \|^2_{L^2(\Gamma)} + \frac{1}{2}\| f \|^2_{L^2(\Omega_d)} ,$$
where $C_2 =  k_*^2 \|\epsilon_* \|_{L^{\infty}(\Omega_d)}   + 1$.  Adding  $\frac{1}{2} \| w \|^2_{L^2(\Omega_d)}$ to both sides, we have 
$$ \| w \|^2_{H^1(\Omega_d)} \leq C_3 \| w \|^2_{L^2(\Omega_d)}  + \frac{1}{2 \eps} \| h \|^2_{L^2(\Gamma)} +  \| f \|^2_{L^2(\Omega_d)},$$
where $C_3 = 2(C_2 + 1)$. This concludes the proof.  
\end{proof}

We now prove that the norm of  $h$ is bounded by  $f$ in domain $\Omega^+$.
\begin{lemma}
\label{lemma:bound_h} Let $f \in L^2(\Omega^+)$ satisfy \eqref{eq:f_cond1}, and $h$ be defined by \eqref{eq:h}. Then, there exists a constant $C$ such that
$$  \| h \|_{L^2(\Gamma)}   \leq C \| f \|_{L^2(\Omega^+)},$$
where $C$ is independent of $f$.
\end{lemma}
\begin{proof} 
In the definition of $h$, the polynomial $Q_m(z)$ (for $m \neq 0$) satisfies \eqref{eq:Q_m}.  Multiplying \eqref{eq:Q_m} by $\overline{Q}'_m(z) e^{- 2\gamma_m z}$, integrating from $d$ to $+\infty$, and using integration by parts, we have
$$ |Q'_m(d)|^2 e^{-2 \gamma_m d} = - \int_d^{+\infty} Q'_m \overline{Q}''_m e^{-2 \gamma_m z} dz - \int_d^{+\infty} \overline{Q}'_m P_m e^{-2 \gamma_m z} dz.  $$
Using \eqref{eq:Q_m} again to eliminate $\overline{Q}''_m$ from the above, we have
$$ |Q'_m(d)|^2 e^{-2 \gamma_m d}  + 2 \gamma_m \int_d^{+\infty} |Q'_m|^2  e^{-2 \gamma_m z} dz = - \int_d^{+\infty} \left( \overline{Q}'_m P_m + Q'_m \overline{P}_m \right) e^{-2 \gamma_m z} dz. $$
All terms in the left-hand side are positive, thus
\begin{equation}
\label{eq:lemma_bound_h_eq1}   |Q'_m(d)|^2 e^{-2 \gamma_m d}  \leq    2\int_d^{+\infty} \left| {Q}'_m P_m  \right| e^{-2 \gamma_m z} dz,
   \end{equation}
and 
\begin{eqnarray} 
\label{eq:lemma_bound_h_eq11}
\int_d^{+\infty} |Q'_m|^2  e^{-2 \gamma_m z} dz 
&\leq & \frac{1}{ \gamma_{\dagger}}   \int_d^{+\infty} \left| {Q}'_m P_m  \right| e^{-2 \gamma_m z} dz,  
\end{eqnarray}
where  $\gamma_m \geq \gamma_{\dagger} > 0$ for all $m\neq 0$ as mentioned in Section \ref{sec:propertiesBIC}. Using H\"{o}lder's inequality to the right-hand side of \eqref{eq:lemma_bound_h_eq11}, we have 
$$  \left( \int_d^{+\infty} |Q'_m|^2  e^{-2 \gamma_m z} dz \right)^{1/2} \leq  \frac{1}{\gamma_{\dagger}} \left( \int_d^{+\infty} |P_m|^2  e^{-2 \gamma_m z} dz \right)^{1/2}. $$
Substituting the above into the inequality  \eqref{eq:lemma_bound_h_eq1} yields
\begin{eqnarray}
\label{eq:lemma_bound_h_eq2}  |Q'_m(d)|^2 e^{-2 \gamma_m d} &\leq &  2 \left( \int_d^{+\infty} |P_m|^2  e^{-2 \gamma_m z} dz \right)^{1/2} \left( \int_d^{+\infty} |Q'_m|^2  e^{-2 \gamma_m z} dz \right)^{1/2}  \\
& \leq & \frac{2}{\gamma_{\dagger}} \int_d^{+\infty} |P_m|^2  e^{-2 \gamma_m z} dz, \nonumber
  \end{eqnarray}
for $m \neq 0.$

Note that
\begin{equation*}
\label{eq:lemma_bound_h_eq4} \| h \|^2_{L^2(\Gamma)} = \int_{-L/2}^{L/2} |h|^2 dy = L \sum\limits_{m\neq 0} |Q'_m(d)|^2 e^{-2 \gamma_m d},
 \end{equation*}
 and
\begin{equation*}
\label{eq:lemma_bound_h_eq3} \| f \|^2_{L^2(\Omega^+)} \geq \| f \|^2_{L^2(\Omega^c_d)}  = L \sum\limits_{m\neq 0} \int_{d}^{+\infty} |P_m|^2 e^{-2 \gamma_m z} dz.
 \end{equation*}
Comparing the right-hand sides of these relations for each $m \neq 0$,   \eqref{eq:lemma_bound_h_eq2} leads to
$$ \| h \|^2_{L^2(\Gamma)} \leq \frac{2}{\gamma_{\dagger}} \| f \|^2_{L^2(\Omega^+)}. $$
This concludes the proof. 
\end{proof}

Next, we prove that the solution of BVP \eqref{eq:BVP2} that is locally orthogonal to the BIC $u_*$ in domain $\Omega_d$ is bounded by the right-hand side inhomogeneous terms.
\begin{lemma}
\label{lemma:bound_inside}  Let $f \in L^2(\Omega^+)$  satisfy  \eqref{eq:f_cond1} and \eqref{eq:f_cond2}, and  $h$ be defined by \eqref{eq:h}.  Then, there exists a constant $C$ such that 
$$ \|  w \|_{L^2(\Omega_d)} \leq C \left[ \| f \|_{L^2(\Omega_d)} + \| h \|_{L^2(\Gamma)}  \right] $$
where $w$ is the solution of  BVP \eqref{eq:BVP2} that satisfies the orthogonal  condition 
$$  (w, u_*)_{L^2(\Omega_d)} = \int_{\Omega_d}  w \overline{u}_* d \mathbf{r} = 0, $$
and $u_*$ is the  BIC. The constant $C$ is independent of $f$ and $h$.
\end{lemma}

The proof follows the proof of Theorem 6 on page 326 of {\cite{evens}}. 
\begin{proof}
If the result is not true, then there  exist sequences $\left\{ f_n \right\}_{n=1}^{\infty} \in L^2(\Omega^+)$ that satisfies \eqref{eq:f_cond1},  related $\left\{ h_n \right\}_{n=1}^{\infty} \in L^2(\Gamma)$,  and $\left\{ w_n \right\}_{n=1}^{\infty} \in H^{1}_{\beta, 0}(\Omega_d)$ that satisfies $(w_n, u_*)_{L^2(\Omega_d)} = 0$,  such that
\begin{eqnarray}
  \quad  \left\{ \begin{array}{ll} \Delta w_n  + k^2_* \epsilon_*  w_n = f_n(y,z),  &  (y,z) \in \Omega_{d}, \\
                          w_n(L/2,z) = e^{ \mathsf{ i } \beta_* L} w_n(-L/2,z), & 0 < z < z, \\ 
                           \frac{\partial w_n}{\partial y}(L/2,z) =  e^{ \mathsf{ i } \beta_* L} \frac{\partial w_n}{\partial y}(-L/2,z), & 0 < z < d, \\
                          \frac{\partial w_n}{\partial z} = 0,  & (y,z) \in \Gamma_0, \\
                          \frac{\partial w_n}{\partial z} = \mathcal{T} w_n + h_n(y), & (y,z) \in \Gamma_d, \end{array} \right.
\end{eqnarray}
in the weak sense, but
$$  \| w_n \|_{L^2(\Omega_d)} \geq n  \left[ \| f_n \|_{L^2(\Omega_d)} + \| h_n \|_{L^2(\Gamma)}  \right], \quad \forall n \geq 1. $$

Without loss of generality, we may assume $\|  w_n \|_{L^2(\Omega_d)} = 1$, then $f_n \to 0$ in $L^2(\Omega_d)$ and $h_n \to 0$ in $L^2(\Gamma)$ as $n \to \infty$. According to Lemma {\ref{lemma:energy}},  sequence $\left\{ w_n \right\}_{n=1}^{\infty} $ is bounded in $H^1_{\beta,0}(\Omega_d)$.  Therefore, there exists a subsequence of $\left\{ w_{n_j} \right\}_{j=1}^{\infty} \subset \left\{ w_n \right\}_{n=1}^{\infty}$ such that $w_{n_j}$ weakly converges to $  w  $ in $ H^1_{\beta, 0}(\Omega_d)$  and  $w_{n_j}$ converges to $  w  $ in  $L^2(\Omega_d)$  as $j \to \infty$ \cite{evens,adams}.
Then $w$ is a solution of BVP \eqref{eq:BVP2} with $f \equiv 0$ and $h \equiv 0$, and satisfies $\| w \|_{L^2(\Omega_d)} = 1$ and $(w, u_*)_{L^2(\Omega_d)} = 0$.

Note that BVP \eqref{eq:BVP2} with $f \equiv 0$ and $h \equiv 0$ has only one non-trivial eigen-solution $u_*$. Thus, $w = C_2 u_*$, where $C_2$ is a constant. Due to condition $(w, u_*)_{L^2(\Omega_d)}    = 0$, we have $C_2 = 0$. This implies that $w \equiv 0$, which is in contradiction to $\| w \|_{L^2(\Omega_d)} = 1$. 
\end{proof}

\noindent {\bf Remark}: Lemmas \ref{lemma:energy}, \ref{lemma:bound_h} and \ref{lemma:bound_inside} imply that there exists a constant $C$ such that 
$$ \|  w \|_{H^1(\Omega_d)} \leq C  \| f \|_{L^2(\Omega^+)}, $$
where $w$ is the solution of  BVP \eqref{eq:BVP2} that satisfies the orthogonal  condition 
$$  (w, u_*)_{L^2(\Omega_d)} = 0.$$

If $w$ is the solution of BVP \eqref{eq:BVP2}  orthogonal to BIC $u_*$ in domain $\Omega_d$, i.e., $w$ satisfies Lemma \ref{lemma:bound_inside},  we  extend $w$ to a solution of BVP \eqref{eq:BVP1} defined on $\Omega^+$. In the following, we show that  the extended solution is bounded  by $f$ in $\Omega_d^c$.
\begin{lemma}
\label{lemma:bound_outside}  
Let $f \in L^2(\Omega^+)$  satisfy  \eqref{eq:f_cond1} -- \eqref{eq:f_cond3}, and  $h$ be defined by \eqref{eq:h}.  Let $w$ be a solution of BVP \eqref{eq:BVP1} whose restriction to $\Omega_d$ satisfies the condition of Lemma \ref{lemma:bound_inside}.  Then there exists a constant $C$ such that
$$ \| w \|_{H^{1}(\Omega^c_d)} \leq C \| f \|_{L^2(\Omega^+)}, $$ 
where  $C$ is independent of $f$.
\end{lemma}

\begin{proof}  
According to \eqref{eq:w_fourier2}, for $z >d$,  $w$ can be expanded as
$$ w = \sum\limits_{m\neq 0} w_m(z) e^{ \mathsf{ i } \beta_m y}, $$
where $w_m(z) = Q_m(z) e^{-\gamma_m z}$ satisfies
$$ w''_m(z) - \gamma_m^2 w_m(z) = f_m(z),$$
with $f_m(z) = P_m(x) e^{-\gamma_m z}$.

For each $m \neq 0$, multiplying the equation of $w_m$ by $\overline{w}_m$ and integrating from $d$ to $+\infty$, we have
$$ \int_{d}^{+\infty} |w'_m|^2 dz + \gamma_m^2 \int_d^{+\infty} |w_m|^2 dz = - \overline{w}_m(d) w'_m(d) - \int_d^{+\infty} \overline{w}_m f_m dz. $$
Since both terms  on the left-hand side are positive, the right-hand side must also be positive. Thus,
\begin{eqnarray} 
\label{eq_lemma_bound2}\int_d^{+\infty} |w_m|^2 dz & \leq & - \frac{1}{\gamma_m} \left[  \overline{w}_m(d) w'_m(d)    +  \int_d^{+\infty} \overline{w}_m f_m dz   \right]   \\
& \leq & - \frac{1}{\gamma_{\dagger}} \left[   \overline{w}_m(d) w'_m(d)    +   \int_d^{+\infty} \overline{w}_m f_m dz   \right]. \nonumber
\end{eqnarray}
Note that 
$$ \| w \|^2_{L^2(\Omega^c_d)}  = L \sum\limits_{m\neq 0} \int_d^{+\infty} |w_m|^2 dz,  $$
$$\int_{-L/2}^{L/2} \overline{w}(y,d) \frac{\partial w}{\partial z}(y,d) dy = L \sum\limits_{m\neq 0} \overline{w}_m(d) w'_m(d),  $$
and
$$ \int_{\Omega^c_d} \overline{w} f d {\bf r} = L \sum\limits_{m\neq 0}  \int_{d}^{+\infty} \overline{w}_m f_m dz. $$
Using these relations, inequality \eqref{eq_lemma_bound2} leads to
\begin{eqnarray}
\label{eq_thm5} \| w \|^2_{L^2(\Omega^c_d)} & \leq  & -  \frac{1}{\gamma_{\dagger}} \left[ \int_{-L/2}^{L/2} \overline{w}(y,d) \frac{\partial w}{\partial z}(y,d) dy   +  \int_{\Omega^c_d} \overline{w} f d {\bf r}  \right]   \\
 & \leq &  \frac{1}{\gamma_{\dagger}} \left[  \left|\int_{-L/2}^{L/2} \overline{w}(y,d) \frac{\partial w}{\partial z}(y,d) dy \right|  + \left| \int_{\Omega^c_d} \overline{w} f d {\bf r} \right| \right]. \nonumber
\end{eqnarray}

Since $w$ is a solution of BVP \eqref{eq:BVP1}, multiplying the governing equation of $w$ by $\overline{w}$ and integrating on $\Omega_d$, we have
$$\int_{-L/2}^{L/2} \overline{w}(y,d) \frac{\partial w}{\partial z}(y,d) dy = \int_{\Omega_d} |\nabla w|^2 d {\bf r} - k_*^2 \int_{\Omega_d}  \epsilon_*  |w|^2 d {\bf r} + \int_{\Omega_d} \overline{w} f d {\bf r}. $$
Thus 
\begin{eqnarray*}
\left| \int_{-L/2}^{L/2} \overline{w}(y,d) \frac{\partial w}{\partial z}(y,d) dy \right| & \leq &  \int_{\Omega_d} |\nabla w|^2 d {\bf r} + k_*^2 \int_{\Omega_d}  \epsilon_*  |w|^2 d {\bf r} + \int_{\Omega_d} | \overline{w} f | d {\bf r} \\
& \leq &  \int_{\Omega_d} |\nabla w|^2 d {\bf r} + k_*^2 \| \epsilon_* \|_{L^{\infty}(\Omega_d)} \int_{\Omega_d}   |w|^2 d {\bf r} \\
&+& \frac{1}{2} \int_{\Omega_d} | w|^2 d {\bf r} + \frac{1}{2} \int_{\Omega_d} | f|^2 d {\bf r} \\
& \leq & C_1 \| w \|^2_{H^{1}(\Omega_d)}  + \frac{1}{2} \| f \|^2_{L^2(\Omega_d)},
\end{eqnarray*}
where $C_1 = \mbox{max}\left\{ 1, k_*^2 \| \epsilon_* \|_{L^{\infty}(\Omega_d)}  + \displaystyle\frac{1}{2} \right\}$.
Furthermore, by the remark of Lemma \ref{lemma:bound_inside}, and noting $\Omega_d \subset \Omega^+$, we have
\begin{equation}
\label{eq:bound_w_wz}
\left| \int_{-L/2}^{L/2} \overline{w}(y,d) \frac{\partial w}{\partial z}(y,d) dy \right|   \leq C_2  \| f \|^2_{L^2(\Omega^+)} ,
\end{equation}
where $C_2$ is a constant.  Inserting the above into  \eqref{eq_thm5},  we have
\begin{equation}
\label{eq_bound1} \| w \|^2_{L^2(\Omega^c_d)}   \leq \frac{C_2}{\gamma_{\dagger}}  \| f \|^2_{L^2(\Omega^+)}  + \frac{1}{\gamma_{\dagger}} \left| \int_{\Omega^c_d} \overline{w} f d {\bf r} \right|.  
 \end{equation}

Applying Cauchy's inequality with $\eps > 0$ to the last term of \eqref{eq_bound1}, we have
$$ \left| \int_{\Omega^c_d} \overline{w} f d {\bf r} \right| \leq \int_{\Omega^c_d}  | \overline{w} f | d {\bf r} \leq \eps \| w \|^2_{L^2(\Omega^c_d)} + \frac{1}{4 \eps} \| f \|^2_{L^2(\Omega^c_d)}.$$
Inserting the above into \eqref{eq_bound1} and  choosing $\eps > 0$ such that $ \eps/\gamma_{\dagger} < 1/2 $, we obtain 
\begin{equation}
\label{eq:bound_w_out}    
\frac{1}{2} \| w \|^2_{L^2(\Omega^c_d)}   \leq  \frac{C_2}{\gamma_{\dagger}}  \| f \|^2_{L^2(\Omega^+)}  + \frac{1}{4 \eps \gamma_{\dagger}} \| f \|^2_{L^2(\Omega^c_d)} \leq C_3 \| f \|^2_{L^2(\Omega^+)} , 
\end{equation}
where $C_3 =  \left(C_2 + \displaystyle\frac{1}{4 \varepsilon} \right)/ \gamma_{\dagger}$.  

Multiplying the governing equation of $w$ by $\overline{w}$ and integrating on $\Omega^c_d$, we have
$$\int_{\Omega^c_d} |\nabla w|^2 d {\bf r}  =   k_*^2 \int_{\Omega^c_d}   |w|^2 d {\bf r} - \int_{-L/2}^{L/2} \overline{w}(y,d) \frac{\partial w}{\partial z}(y,d) dy - \int_{\Omega^c_d} \overline{w} f d {\bf r}. $$
Using \eqref{eq:bound_w_out}, \eqref{eq:bound_w_wz}, and Cauchy's inequality to the above equation, we have
$$ \| \nabla w \|_{L^2(\Omega^c_d)} \leq C_4 \| f \|_{L^2(\Omega^+)},$$
where $C_4$ is a constant.
Combining the above equation with \eqref{eq:bound_w_out}, we have
$$ \|  w \|_{H^1(\Omega^c_d)} \leq C_5 \| f \|_{L^2(\Omega^+)}, $$
where $C_5 = 2C_3 + C_4$. This concludes the proof.
\end{proof}

Lemmas \ref{lemma:bound_inside} and \ref{lemma:bound_outside} indicate that the solution of BVP \eqref{eq:BVP1} that is locally orthogonal to BIC $u_*$ in $\Omega_d$  is bounded by the inhomogeneous term $f$ in $\Omega_d$ and $\Omega^c_d$, respectively. Consequently, this solution is bounded by $f$ throughout the entire domain $\Omega^+$. We now prove that the solution of  BVP \eqref{eq:BVP1}, orthogonal to BIC $u_*$ in $\Omega^+$, is also  bounded by $f$.

\begin{theorem}
\label{theorem:bound_orth} Let $f \in L^2(\Omega^+)$ satisfy  \eqref{eq:f_cond1} - \eqref{eq:f_cond3}. Then, there exists a constant $C$ such that 
$$ \| w \|_{H^1(\Omega^+)} \leq C \| f \|_{L^2(\Omega^+)}$$
where $w$ is the solution of BVP \eqref{eq:BVP1} that satisfies
$$ (w,u_*)_{L^2(\Omega^+)}  = \int_{\Omega^+} w \overline{u}_* d {\bf r} = 0, $$
where $u_*$ is the  BIC, and $C$ is independent of $f$.
\end{theorem}

\begin{proof} 
Let $\tilde{w}$ be the solution of BVP \eqref{eq:BVP1} that satisfies Lemma  \ref{lemma:bound_outside}, then
$$ \|  \tilde{w} \|_{H^1(\Omega^c_d)} \leq C_1 \| f \|_{L^2(\Omega^+)}.$$
On the other hand, the remark of Lemma \ref{lemma:bound_inside} leads to
$$ \| \tilde{w} \|_{H^1(\Omega_d)} \leq C_2 \| f \|_{L^2(\Omega^+)}. $$
 Thus
\begin{equation}
\label{eq:theorem_bound_orth_1}   \|  \tilde{w} \|_{H^1(\Omega^+)} \leq C_3 \| f \|_{L^2(\Omega^+)},
   \end{equation}
where $C_3 = C_1 + C_2.$

Let $w = \tilde{w} + C_4 u_*$, where
$$ C_4 = - \int_{\Omega^c_d}{ \tilde{w} \overline{u}_*} d {\bf r}. $$
Due to  $ (\tilde{w}, u_*)_{L^2(\Omega_d)} = 0 $ and the  the normalization condition  $ \| u_* \|_{L^2(\Omega^+)} = 1$,  it is easy to verify that  $w$ is a solution of BVP \eqref{eq:BVP1} and satisfies the condition $(w,u_*)_{L^2(\Omega^+)} = 0$. 
Thus,
\begin{eqnarray*}
 |C_4|  \leq \int_{\Omega^c_d} \left| { \tilde{w} \overline{u}_*} \right| d {\bf r} \leq \left\| \tilde{w} \right\|_{L^2(\Omega^c_d)}  \left\| u_* \right\|_{L^2(\Omega^c_d)}  \leq \left\| \tilde{w} \right\|_{H^1(\Omega^+)},
 \end{eqnarray*}
and therefore,
$$ \| w \|_{H^1(\Omega^+)} \leq   \| \tilde{w} \|_{H^1(\Omega^+)} + |C_4| \|  u_* \|_{H^1(\Omega^+)} \leq  C \| f \|_{L^2(\Omega^+)},$$
where $C = C_3 \left[1 + \|  u_* \|_{H^1(\Omega^+)} \right].$  
\end{proof}

\section{Robustness of BICs}

 In a previous work {\cite{yuan17ol}},  we used a
 perturbation method to argue that generic BICs in  periodic structures
 are robust against structural perturbations  that  preserve the periodicity and symmetries of the original
 structure. In this section, we give  a  mathematical
 proof for the robustness theory. First, we summarize the assumptions as follows. 
\begin{itemize}
\item[A1:] The dielectric function $\epsilon_*(\mathbf{r}) \in L^{\infty}(\Omega)$ of the original structure is real, periodic in $y$ with period $L$,   has reflection symmetries in both $y$ and $z$, and $\epsilon_* = 1$ for $|z| > d$.
\item[A2:] The original structure has a BIC $\{ u_*(\mathbf{r}) = \phi_*(\mathbf{r}) e^{ \mathsf{ i } \beta_* y}, \beta_*, k_* \}$ which is nondegenerate,  satisfies \eqref{eq:cond_k_beta} and 
\begin{equation}
\label{eq:genericBIC}
(\partial_y u_*, v_*)_{L^2(\Omega)} = \int_{\Omega} \overline{v}_* \partial_y u_* d \mathbf{r}   \neq 0,
\end{equation}
where $v_*(\mathbf{r}) = \varphi_*(\mathbf{r}) e^{ \mathsf{ i } \beta_* y} $ is the corresponding diffraction solution defined in Section~\ref{sec:diffraction_solution}. $u_*$ is scaled to be $\mathcal{PT}$-symmetric in $y$ and normalized such that $\| u_* \|_{L^2(\Omega)} = 1$.
\item[A3:] The perturbed structure has a dielectric function given by
\begin{equation}
\label{eq:pert_eps} \epsilon(\mathbf{r}) = \epsilon_*(\mathbf{r}) + \delta F(\mathbf{r}),
\end{equation}
where $\delta $ is the amplitude of the perturbation, and $F \in L^{\infty}(\Omega)$ is the perturbation profile, and it is real and periodic in $y$ with period $L$. In addition, $F$ has reflection symmetries in both $y$ and $z$ directions,  and satisfies $F(\mathbf{r})=0$ for $|z|>d$.
\end{itemize}

The robustness theory of the BIC can be stated as the following theorem.
\begin{theorem}
\label{theorem:robustness}
    Assume that the original structure with $\epsilon_*(\mathbf{r})$ satisfies A1, and has a BIC that satisfies A2. Then,  there is a $\delta_* > 0$, such that for any $\delta \in [-\delta_*, \delta_*]$,  the perturbed structure with dielectric function satisfying A3 has a BIC $\{ u(\mathbf{r}) = \phi(\mathbf{r}) e^{ \mathsf{ i } \beta y}, \beta, k \}$ with $u(\mathbf{r})$ near $u_*(\mathbf{r})$, $\beta$ near $\beta_*$ and $k$ near $k_*$. 
\end{theorem}

In the following, we give this theorem a constructive proof. It contains two parts. First, we construct the perturbed BIC  as a power series of $\delta$ in Section~\ref{sec:powerseries}. Then, we show that the series converges if $\delta$ is sufficiently small (in absolute value) in Section~\ref{sec:convergence}. 
 
\subsection{Power series}
\label{sec:powerseries}
For the perturbed structure, we seek a BIC $\{ u(\mathbf{r}) = \phi(\mathbf{r}) e^{ \mathsf{ i } \beta y}, \beta, k \}$ by expanding  $\phi$,  $\beta$  and $k^2$  in power series of $\delta$ as:
\begin{eqnarray}
  \label{eq:expphi}  \phi  = \sum_{j=0}^{\infty} \phi_j \delta^j, \quad  \beta = \sum_{j=0}^{\infty} \tau_j \delta^j, \quad
                     k^2  = \sum_{j=0}^{\infty} g_j \delta^j, 
\end{eqnarray}
where $\phi_0 = \phi_*, g_0 = k_*^2$ and $\tau_0 = \beta_*$.  
Substituting \eqref{eq:pert_eps} and \eqref{eq:expphi}  into \eqref{eq:BICs_phi} and comparing the coefficient of $\delta^j$ for $j \geq 1$,  we obtain the following BVPs for $\phi_j$: 
\begin{eqnarray}
\label{eq:phi_j}    \left\{ \begin{array}{ll} \mathcal{L}_* \phi_j    = \tau_j B_1(\mathbf{r}) + g_j B_2(\mathbf{r}) - D_j(\mathbf{r}),  &  {\bf r} \in \Omega, \\
                          \phi_j(L/2,z) = \phi_j(-L/2,z), & z \in \mathbb{R},\\ 
                          \partial_y \phi_j (L/2,z) =  
                             \partial_y \phi_j (-L/2,z),  & z \in \mathbb{R},\\
                          \phi_j({\bf r})  \to 0 \quad \mbox{as} \quad   z
                             \to  \pm \infty, &| y | < L/2,
                          \end{array} \right.
\end{eqnarray}
where $\mathcal{L}_* =  \Delta + 2 \textsf{i} \beta_* \partial_y + k_*^2 \epsilon_*(\mathbf{r}) - \beta_*^2$, $  B_1(\mathbf{r}) = 2 ( \beta_* \phi_* - \textsf{i} \partial_y \phi_*), B_2(\mathbf{r}) = - \epsilon_* \phi_* $, and
\begin{eqnarray}
\label{eq:Dj}  
D_{j} &=&\sum\limits_{n=1}^{j}  g_{n-1} F \phi_{j-n}    - \sum_{n=2}^{j}   \sum\limits_{m=1}^{n-1} \tau_m \tau_{n-m}  \phi_{j-n} \\
& & + \sum\limits_{n=1}^{j-1} \left(   g_n \epsilon_*  - 2   \beta_* \tau_n  + 2 {\mathsf{i}} \tau_n \partial_y  \right) \phi_{j-n}. \nonumber
\end{eqnarray}
Note that $D_j$ is only related to $\phi_n, \tau_n$, and $g_n$ for $n \leq j - 1$.

We show that for each $j \geq 1$, there are real and unique $\tau_j$ and $g_j$  such that  BVP \eqref{eq:phi_j} has a solution $\phi_j$.
We have the following theorem.
\begin{theorem}
\label{theorem:exist_phij}
    If assumptions A1, A2, and A3 are true, then there exist two real sequences $\{ \tau_j \}_{j=1}^{\infty}$ and $\{ g_j \}_{j=1}^{\infty}$, such that for each $j \geq 1$, BVP \eqref{eq:phi_j} has a solution $\phi_j$ that satisfies $$\| \phi_j \|_{H^1(\Omega)} \leq C \|  D_j\|_{L^2(\Omega)},$$
    where $C$ is a constant independent of $\tau_j, g_j$ and $\phi_j$ for $j \geq 1$.
\end{theorem}

\begin{proof}

Without loss of generality, we assume  that $u_*$ and $v_*$ are even in $z$. Thus, $B_1, B_2$ and $D_j$ are  even in $z$. Let $w_j = \phi_j e^{\textsf{i} \beta_* y}$ and $f_j = ( \tau_j B_1 + g_j B_2 - D_j ) e^{\textsf{i} \beta_* y} $ for each $j \geq 1$, then BVP \eqref{eq:phi_j}  is equivalent to 
\begin{eqnarray}
\label{eq:w_j}    \left\{ \begin{array}{ll} \left[ \Delta + k_*^2 \epsilon_*(\mathbf{r}) \right] w_j    = f_j(\mathbf{r}),  &  {\bf r} \in \Omega^+, \\
                          w_j(L/2,z) = e^{ \mathsf{ i } \beta_* L} w_j(-L/2,z), & z > 0,\\ 
                          \partial_y w_j (L/2,z) =  e^{ \mathsf{ i } \beta_* L} \partial_y w_j (-L/2,z),  & z > 0,\\
                          \partial_z w_j(y,0) = 0, & | y | < L/2,\\
                          w_j({\bf r})  \to 0 \quad \mbox{as} \quad   z
                             \to  \pm \infty, & | y | < L/2.
                          \end{array} \right.
\end{eqnarray}
We only need to show that  BVP \eqref{eq:w_j} for each $j \geq 1$ has a solution $w_j$ that is bounded by $f_j$ in $\Omega^+$.
This is equivalent to show that BVP \eqref{eq:BVP1} with $f = f_j$ for each $j \geq 1$ has a solution $w_j$ that decays to zero as $z \to +\infty$ and satisfies $ \| w_j \|_{H^1(\Omega)} \leq C \|  D_j\|_{L^2(\Omega)}.$

We first prove the solvability.
By Theorem \ref{theorem:decay}, we only need to choose real $\tau_j$ and $g_j$ such that  $f_j$ satisfies conditions \eqref{eq:f_cond1} - \eqref{eq:f_cond3} for all $j \geq 1$.  
For $j = 1$, we have $D_1(\mathbf{r}) = k_*^2 F \phi_*$ and $f_1(\mathbf{r}) = - 2 {\mathsf{i}} \tau_1 \partial_yu_* - g_1 \epsilon_* u_* - k_*^2 F u_*$. Using the  expansion \eqref{eq:BIC_fourier} of the BIC, we obtain
$$f_1 = \sum\limits_{m\neq 0} (2\beta_m \tau_1 - 2g_1) e^{ \mathsf{ i } \beta_m y - \gamma_m z}, \quad z > d.$$
Thus, $f_1$ satisfies the condition \eqref{eq:f_cond1} for any $\tau_1$ and $g_1$. We choose $\tau_1$ and $g_1$ such that 
\begin{eqnarray}
    \label{eq:constriant_w1}  (f_1, u_*)_{L^2(\Omega)} =0, \quad (f_1, v_*)_{L^2(\Omega)}  = 0.
\end{eqnarray}
This leads to a  linear system for $\tau_1$ and $g_1$ as
\begin{equation}
  \label{eq:2by2_n1}
  A 
\begin{bmatrix} \tau_1 \cr g_1 \end{bmatrix} 
=  \begin{bmatrix} \left(u_*,  D_1 e^{\mathsf{i} \beta_* y} \right)_{L^2(\Omega)} \cr \left(v_*,  D_1 e^{\mathsf{i} \beta_* y} \right)_{L^2(\Omega)}\end{bmatrix} , 
\end{equation}
where
$$A =  - \begin{bmatrix}  2 {\mathsf{i} } (\partial_y u_*, u_*)_{L^2(\Omega)} &   (\epsilon_* u_*, u_*)_{L^2(\Omega)} \cr  2 {\mathsf{i} } (\partial_y u_*, v_*)_{L^2(\Omega)} & 0 \end{bmatrix}.$$
Note that $u_*, v_*,$ and $ f_1$ are even in $z$, conditions \eqref{eq:constriant_w1} are equivalent to  
$$(f_1, u_*)_{L^2(\Omega^+)}  = (f_1, v_*)_{L^2(\Omega^+)}  =0.$$ 
Thus, we can solve $\tau_1$ and $g_1$ from \eqref{eq:2by2_n1}, and $f_1$ satisfies  \eqref{eq:f_cond2} and \eqref{eq:f_cond3}. 
Since condition \eqref{eq:genericBIC} ensures that $ (\partial_y u_*, v_*)_{L^2(\Omega)}  \neq 0 $, and $u_*, v_*, {\mathsf{i}}\partial_y u_*$ and $ D_1$ are all $\mathcal{PT}$-symmetric in $y$, the linear system \eqref{eq:2by2_n1} is invertible and real. Thus,  $\tau_1$ and $g_1$  are real, and  BVP \eqref{eq:w_j} for $j=1$ admits a solution $w_1$ satisfying $w_1 \to 0$ as $z \to +\infty$.   In addition, $w_1$ is $\mathcal{PT}$-symmetric in $y$ and
\begin{equation}
    \label{eq:w1_fourier} w_1(\mathbf{r}) = \sum\limits_{m\neq 0}  Q_m^{(1)}(z)  e^{ \mathsf{ i } \beta_m y - \gamma_m z}, \quad z > d,
\end{equation}
where $Q_m^{(1)}(z)$ are polynomials of order $1$.

For $j \geq 2$, it is easy to verify  that $f_j$ can be written as 
\begin{equation*}
    \label{eq:fj_fourier} f_j(\mathbf{r}) = \sum\limits_{m\neq 0}  P_m^{(j-1)}(z)  e^{ \mathsf{ i } \beta_m y - \gamma_m z}, \quad z > d,
\end{equation*}
where $P_m^{(j-1)}(z)$ are polynomials of order $j-1$. Thus, $f_j$ satisfies \eqref{eq:f_cond1}. 
We choose $\tau_j$ and $g_j$ such that $ (f_j, v_*)_{L^2(\Omega)} = ( f_j, v_*)_{L^2(\Omega)}  =0$. Then, $f_j$ satisfies \eqref{eq:f_cond2} and \eqref{eq:f_cond3}.  We can check that $f_j$ and $D_j$ are $\mathcal{PT}$-symmetric in $y$. Thus, real $\tau_j$ and $g_j$ can be solved from the linear system
\begin{equation}
  \label{eq:2by2}
  A 
\begin{bmatrix} \tau_j \cr g_j \end{bmatrix} 
=  \begin{bmatrix} \left( D_j e^{\mathsf{i} \beta_* y}, u_* \right)_{L^2(\Omega)} \cr \left( D_j e^{\mathsf{i} \beta_* y}, v_* \right)_{L^2(\Omega)}\end{bmatrix}. 
\end{equation}
 Similarly,  $w_j$ is $\mathcal{PT}$-symmetric and
\begin{equation}
    \label{eq:wj_fourier} w_j(\mathbf{r}) = \sum\limits_{m\neq 0}  Q_m^{(j)}(z)  e^{ \mathsf{ i } \beta_m y - \gamma_m z}, \quad z > d,
\end{equation}
where $Q_m^{(j)}$ are polynomials of $z$ of order $j$.

We next prove that $w_j$ and $\phi_j$ are bounded by $D_j$ for $j \geq 1$. Note that matrix $A$ is invertible and  related to the BIC $u_*$ and the corresponding diffraction solution $v_*$. From linear system \eqref{eq:2by2}, we have
\begin{equation}
\label{eq:bound_tau_k} |\tau_j| \leq C_1 \| D_j \|_{L^2(\Omega)}, \quad  |g_j| \leq C_1 \| D_j \|_{L^2(\Omega)},
 \end{equation}
where $C_1$ is a constant related to matrix $A$. By Theorem \ref{theorem:bound_orth}, if we choose $w_j$ such that $(w_j, u_*)_{L^2(\Omega)} = 2(w_j, u_*)_{L^2(\Omega^+)} = 0$, i.e.,  $(\phi_j, \phi_*)_{L^2(\Omega)}  = 0$, then 
$$ \| w_j \|_{H^1(\Omega)} = 2 \| w_j \|_{H^1(\Omega^+)} \leq 2 \| f_j \|_{L^2(\Omega^+)} = \| f_j \|_{L^2(\Omega)}. $$
Noting that $f_j = ( \tau_j B_1 + g_j B_2 - D_j ) e^{\textsf{i} \beta_* y} $ and using \eqref{eq:bound_tau_k}, we have
$$    \| w_j \|_{H^1(\Omega)} \leq   C_2 \| D_j \|_{L^2(\Omega)}, $$
where $C_2 = C_1\left( \| B_1 \|_{L^2(\Omega)} + \| B_2 \|_{L^2(\Omega)}\right) + 1.$  Since $\phi_j = w_j e^{ -\mathsf{i} \beta_* y }$, we have
$$ \| \nabla \phi_j \|_{L^2(\Omega)} \leq (1 + 2 |\beta_*|) \| \nabla w_j \|_{L^2(\Omega)} +  \beta_*^2 \|  w_j \|_{L^2(\Omega)} \leq (1 + |\beta_*|)^2 \|  w_j \|_{H^1(\Omega)}.  $$
Thus
$$ \|  \phi_j \|_{H^1(\Omega)} \leq \left[ (1 + |\beta_*|)^2 + 1 \right] \|  w_j \|_{H^1(\Omega)} \leq C \| D_j \|_{L^2(\Omega)},$$
where $C = C_2 \left[ (1 + |\beta_*|)^2 + 1 \right].$
\end{proof}

\subsection{Convergence}
\label{sec:convergence}
In this section, we prove that all power series in \eqref{eq:expphi}  are convergent for sufficiently small $\delta$.  
We begin by proving the following lemma.
\begin{lemma}
\label{lemma:new_series} Suppose assumptions A1, A2, and A3 are true, and let $\{ \tau_j \}_{j=1}^{\infty}$, $\{ g_j \}_{j=1}^{\infty}$ and $\{\phi_j\}_{j=1}^{\infty}$ be the sequences obtained in Theorem \ref{theorem:exist_phij}. Then there exists a constant $C_* > 1$ such that sequence $\{d_j\}_{j=0}^{\infty}$, where $d_0 = 2 C_*^2$ and $d_j = 2 C_*^2 \| D_j \|_{L^2(\Omega)}$ for $j \geq 1$, satisfies
$$ d_j  \leq  d_0 \sum_{n=1}^{j} d_n d_{j-n}  + \sum_{n=2}^{j} \sum_{m=1}^{n-1} d_m d_{n-m} d_{j-n} .$$
\end{lemma}

\begin{proof}
For each $j \geq 1$, using \eqref{eq:Dj}, we have
\begin{eqnarray}
\| D_j \|_{L^2(\Omega)} & \leq &  \sum_{n=1}^{j} |g_{n-1}| \| F\|_{L^{\infty}(\Omega)} \| \phi_{j-n} \|_{L^2(\Omega)}  \\
&+& \sum_{n=2}^{j} \sum\limits_{m=1}^{n-1} |\tau_m| |\tau_{n-m}| \| \phi_{j-n} \|_{L^2(\Omega)} \nonumber  \\
&+& \sum_{n=1}^{j-1} \left(  |g_n| \| \epsilon_* \|_{L^{\infty}(\Omega)} + 2 |\beta_*| |\tau_n| + 2 |\tau_n| \right) \| \phi_{j-n} \|_{H^1(\Omega)} . \nonumber 
\end{eqnarray}
Let $$  C_* = \max \left\{2, C, C_1, \|\epsilon_*\|_{L^{\infty}(\Omega)}, \|F\|_{L^{\infty}(\Omega)},  2|\beta_*|, k_*^2 \right\},$$
where $C$ and $C_1$ are the constants in Theorem \ref{theorem:exist_phij}.
Then, for $j \geq 1$, we have
$$ |\tau_j | \leq C_* \| D_j \|_{L^2(\Omega)}, \quad  |g_j | \leq  C_* \| D_j \|_{L^2(\Omega)}, \quad \| \phi_j \|_{H^1(\Omega)} \leq C_* \| D_j \|_{L^2(\Omega)}.$$
Let
 $\| D_0 \|_{L^2(\Omega)} = 1$ and note that $\| \phi_* \|_{L^2(\Omega)} = \| u_* \|_{L^2(\Omega)} = 1$, then
\begin{eqnarray}
\| D_j \|_{L^2(\Omega)} & \leq &  C_*^3 \sum_{n=1}^{j} \| D_{n-1} \|_{L^2(\Omega)} \| D_{j-n} \|_{L^2(\Omega)}  \\ 
&+& C_*^3 \sum_{n=2}^{j} \sum\limits_{m=1}^{n-1} \| D_{m} \|_{L^2(\Omega)} \| D_{n-m} \|_{L^2(\Omega)} \| D_{j-n} \|_{L^2(\Omega)}  \nonumber \\
&+& 3 C_*^3 \sum_{n=1}^{j-1} \| D_{n} \|_{L^2(\Omega)}  \| D_{j-n} \|_{L^2(\Omega)}  \nonumber \\
&\leq& C_*^3 \sum_{n=1}^{j} \| D_{n-1} \|_{L^2(\Omega)} \| D_{j-n} \|_{L^2(\Omega)}  \nonumber \\
&+& 4 C_*^3 \sum_{n=2}^{j} \sum\limits_{m=1}^{n-1} \| D_{m} \|_{L^2(\Omega)} \| D_{n-m} \|_{L^2(\Omega)} \| D_{j-n} \|_{L^2(\Omega)}. \nonumber
\end{eqnarray}
Multiplying  $2 C_*^2$ to both sides of above inequality, and letting $d_j = 2 C_*^2 \| D_j \|_{L^2(\Omega)} $ for $j \geq 0$, we have
$$ d_j  \leq  d_0 \sum_{n = 1}^{j} d_{n-1} d_{j-n}  + \sum_{n=2}^{j} \sum_{m=1}^{n-1} d_m d_{n-m} d_{j-n}, $$
for $j \geq 1$.
\end{proof}

\begin{lemma}
\label{lemma:convergence} Suppose assumptions A1, A2, and A3 are true, and let  $\{ \tau_j \}_{j=1}^{\infty}$, $\{ g_j \}_{j=1}^{\infty}$ and $\{\phi_j\}_{j=1}^{\infty}$ be the sequences obtained in Theorem \ref{theorem:exist_phij}. Then, there exist a $\delta_* > 0$ such that all power series in \eqref{eq:expphi} converge for $|\delta| \leq \delta_*$.
\end{lemma}

\begin{proof}
By Lemma \ref{lemma:new_series},  we have 
$$\| \phi_j \|_{H^1(\Omega)} \leq d_j, \quad |\tau_j| \leq d_j, \quad |g_j| \leq d_j$$
for $j \geq 1$. We only show that the power series $ \sum\limits_{j=0}^{\infty} d_j \delta^j$ converges for any sufficiently small $\delta$.

Define a sequence $\{e_j\}_{j=0}^{\infty}$ by  $e_0 = d_0 = 2 C_*^2$, and
$$ e_j  =    e_0 \sum_{n=1}^{j} e_{n-1} e_{j-n}  + \sum_{n=2}^{j} \sum_{m=1}^{n-1} e_m e_{n-m} e_{j-n} $$
for $j \geq 1$, then 
$$e_j \geq d_j \geq 0.$$
We only need to prove that the power series $\sum\limits_{j=0}^{\infty} e_j \delta^j$ converges. Note that  $\{e_j\}_{j=0}^{\infty}$ is an increasing sequence and $e_1 =e^3_0= d_0^3$, then for $j \geq 2$ we have
\begin{eqnarray} 
\label{eq:ej}
e_j &\leq& 3 \sum_{n=0}^{j-2} \sum\limits_{m=0}^n e_{m+1} e_{n-m} e_{j-n-1} \leq  3 \sum_{n=0}^{j-2} \sum\limits_{m=0}^n e_{m+1} e_{n-m + 1} e_{j-n-1} \\
& \leq & 9 \sum_{n=0}^{j-2} \sum\limits_{m=0}^n e_{m+1} e_{n-m + 1} e_{j-n-1}. \nonumber
\end{eqnarray}
We further define a sequence $\{p_j\}_{j=0}^{\infty}$ by $p_0 = 3 e_1$ and 
$$ p_j = \sum_{n=0}^{j-1} \sum_{m=0}^n p_m p_{n-m} p_{j-1-n}, \quad j \geq 1.$$
Using \eqref{eq:ej}, we have   $p_j \geq 3 e_{j+1}$ for  $j \geq 0$. 
It is easy to check that $p_j = p_0^{2j+1} A_j(3,1)$ for $j \geq 0$, where 
$$ A_j(3,1) = \frac{1}{2j+1} \left( \begin{matrix} 3j \\ j \end{matrix} \right) \sim \frac{\sqrt{3}}{4\sqrt{\pi}} j ^{-3/2} \left( \frac{27}{4} \right)^j$$
are the Fuss-Catalan numbers \cite{qi18}. Therefore, the power series $\sum\limits_{j=0}^{\infty} p_j \delta^j$ converges for
$$|\delta| \leq \delta_* = \frac{4}{27 p_0^2} = \frac{1}{3888 C_*^{12}},$$
where $C_*$ is defined in Lemma \ref{lemma:new_series}. This concludes the proof.
\end{proof}

Theorem \ref{theorem:robustness} is a direct consequence of Theorem \ref{theorem:exist_phij} and Lemma \ref{lemma:convergence}. Theorem \ref{theorem:exist_phij} implies that all terms of power series in \eqref{eq:expphi} are solvable, and Lemma \ref{lemma:convergence} shows that the series  converge for sufficiently small $|\delta|$. Therefore, for any sufficiently small $\delta$, the series in \eqref{eq:expphi} give rise to a BIC in the perturbed structure with a dielectric function given by \eqref{eq:pert_eps}.

\section{Conclusion}
In this paper, we  presented a rigorous theory on the robustness of generic BICs in 2D dielectric structures with a single periodic direction. In our earlier work \cite{yuan17ol}, the BICs in perturbed structures were formally constructed as power series of the perturbation amplitude $\delta$ without showing the convergence of the power series. Here, we formulated the  problem in appropriate function spaces, established the solvability conditions and norm estimates for every term in the power series, and proved that the series indeed converge for sufficiently small $|\delta|$. This provides a rigorous mathematical justification for the formal  theory of \cite{yuan17ol}, confirming the robustness of generic BICs.

Our analysis is restricted to 2D structures with a single periodic direction, nondegenerate and generic BICs with a single radiation channel, symmetry-preserving perturbations, and TE polarization. However, the overall framework can be extended to rigorously analyze the robustness of BICs in 3D structures with two periodic directions \cite{yuan21rob}, and in other structures, such as waveguides with lateral leakage channels \cite{nan24opex} and rotational symmetric periodic waveguides. It can also be used to study the parametric dependence theory of BICs with multiple radiation channels or under symmetry-breaking perturbations \cite{lijun20pd1,amgad23pra}. Moreover, the methodology can be extended to  rigorously analyze the bifurcation  of non-generic BICs \cite{nan24ol}.

\end{document}